\documentclass{amsart}
\usepackage{graphicx}
\usepackage{amscd}
\usepackage{amsmath}
\usepackage{amsfonts}
\usepackage{amssymb}
\theoremstyle{plain}
\newtheorem{theorem}{Theorem}[section]
\newtheorem{corollary}[theorem]{Corollary}
\newtheorem{lemma}[theorem]{Lemma}
\newtheorem{proposition}[theorem]{Proposition}
\theoremstyle{definition}
\newtheorem{example}[theorem]{Example}

\newtheorem{definition}[theorem]{Definition}

\newtheorem{remark}[theorem]{Remark}
\theoremstyle{remark}

\begin{document}
\title{SP-rings with zero-divisors}

\author{Malik Tusif Ahmed and Tiberiu Dumitrescu}
\address{Abdus Salam School of Mathematical Sciences GCU Lahore, Pakistan}
\email{tusif.ahmed@sms.edu.pk, tusif.ahmad92@gmail.com (Ahmed)}
\address{Facultatea de Matematica si Informatica,University of Bucharest,14 A\-ca\-de\-mi\-ei Str., Bucharest, RO 010014,Romania}
\email{tiberiu@fmi.unibuc.ro, tiberiu\_dumitrescu2003@yahoo.com (Dumitrescu)}

\begin{abstract}
We characterize the commutative rings whose ideals (resp. regular ideals) are products of radical ideals.
\end{abstract}

\thanks{2000 Mathematics Subject Classification: Primary 13A15, Secondary 13F05.}
\keywords{radical factorization, ring with zero divisors, almost multiplication ring}

\maketitle


\section{Introduction}
In \cite{VY},  
Vaughan and Yeagy introduced and studied  {\em SP-domains}, that is,  integral domains whose ideals are products of radical ideals (here ``SP'' stands for {\em semiprime ideals} another name for  radical ideals). A Dedekind domain is clearly SP. An SP-domain $D$ is almost-Dedekind (i.e. $D_M$ is a discrete rank one valuation domain for each maximal ideal $M$ of $D$) cf.  \cite[Theorem 2.4]{VY}). Examples of  SP-domains which are not Dedekind and of  almost Dedekind domains which are not SP-domains are given in \cite{VY}.
The study of SP-domains was continued by Olberding in \cite{O} who gave several characterizations for SP-domains  inside the class of almost Dedekind domains  \cite[Theorem 2.1]{O}. He also proved that given a Boolean topological space $X$, there exists an SP-domain $D$ such that $Max(D)$ is homeomorphic to $X$  \cite[Theorem 3.2]{O}.

In this paper we  study two generalizations of  SP-domain concept in the  setup of (commutative) rings with zero-divisors. In Section 2 we consider the class of rings whose  regular ideals are products of radical ideals (called by us {\em SP-rings}). As usual, a {\em regular ideal} is an ideal containing a regular element (i.e. non-zerodivisor). Denote  the set of regular elements of a ring $A$ by $Reg(A)$.
A Dedekind ring, i.e. a ring whose regular ideals are products of prime ideals  is clearly an SP-ring. For technical convenience, in Section 2, we restrict our study to Marot rings (i.e. to rings whose regular ideals are generated by regular elements, see \cite[page 31]{H}). 

A  localization  of an SP-ring  with respect to a multiplicative set consisting of regular elements is an SP-ring (Proposition \ref{1}), but a  localization of an SP-ring at some prime ideal is not necessarily an SP-ring
(Remark \ref{13}). For the idealization of a module, the SP-ring property is characterized in Proposition \ref{6}.

If $A$ is an SP-ring, then it  is an N-ring, that is, $A_{(M)}$ is a  discrete rank one Manis valuation ring for each maximal regular ideal $M$ of $A$ (Theorem \ref{10}). This result extends \cite[Theorem 2.4]{VY} recalled  above.
Here $A_{(M)}$ is the regular localization of $A$ at $M$, i.e. the fraction ring  $A_T$ where $T=Reg(A)\cap (A-M)$.
In particular, the regular-Noetherian SP-rings are exactly the Dedekind rings
(Corollary \ref{19}).
Theorem \ref{7} characterizes SP-rings among N-rings; it is a natural extension of \cite[Theorem 2.1]{O} to SP-rings.
If $A$ is a von Neuman regular ring containing $\mathbb{Q}$, then $A[X]$ is an SP-ring (Corollary \ref{18}). 
 An infinite direct product of rings is an SP-ring iff all factors are SP-rings and at most finitely many of them are not total quotient rings (Corollary \ref{12}).

In Section 3, we extend SP-domains in another  way: we consider the class of  rings whose  ideals are products of radical ideals (called by us {\em SSP-rings} for strongly SP-rings). A  {\em ZPI ring} (i.e. a ring whose proper ideals are products of prime ideals) is  clearly an SSP-ring. In Theorem \ref{11} we prove that
 an  SSP-ring $A$ is an almost multiplication ring, i.e. for every prime ideal $P$, the localization $A_P$ is a  discrete rank one valuation domain  or special primary ring ($=$ a local ring  whose proper ideals are powers of its maximal ideal).
 This is another extension of \cite[Theorem 2.4]{VY}.
 Consequently, an SSP-ring which has only finitely many minimal prime ideals is  a finite direct product of special primary rings and SP-domains
(Corollary \ref{15}). 
In particular, the  Noetherian SSP-rings are exactly the ZPI-rings (Corollary \ref{14}).
A multiplication module over a von Neuman regular ring has SSP idealization (Proposition \ref{16}).
Theorem \ref{17} characterizes  SSP-rings among   almost multiplication rings  whose localizations at  maximal ideals are discrete valuation domains.  
The main results of this paper are Theorems \ref{10}, \ref{7} and \ref{11}.

All rings considered in this paper are commutative and unitary. For any undefined terminology or for basic facts about rings with zerodivisors see \cite{G}, \cite{H} and \cite{LM}. We use the common practice of abbreviating ``if and only if'' by ``iff''.

\section{SP-rings}

All rings in this section are {\em Marot rings}, that is,  their regular ideals are generated by regular elements (see \cite[page 31]{H}).
We extend the concept of SP-domain to rings with zero-divisors.

\begin{definition}
A ring  is called an {\em SP-ring} if its regular ideals are products of radical ideals.
\end{definition}

Thus a domain is an SP-ring iff it is an SP-domain. Total quotient rings and Dedekind rings are clearly SP-rings. Recall that a {\em Dedekind ring} is a ring whose regular ideals are products of prime ideals (see \cite[Theorem 17]{Gf}).
It is easy to see that a direct product of two rings is an SP-ring iff the factors are SP-rings.

 For  be a prime ideal $P$ of a ring $A$, denote (as usual) by $A_{(P)}$  the regular localization of $A$ at $P$, i.e. the fraction ring  of $A$ with denominators in $Reg(A)\cap (A-P)$. The ring embeds naturally in the total quotient ring $K=A_{Reg(A)}$ of $A$; besides, if $A$ is a Marot ring, then $A_{(P)}$ coincides with the {\em large quotient ring} $A_{[P]}=\{ x\in K\mid sx\in A$ for some $s\in A-P\}$, see \cite[page 28 and Theorem 7.6]{H}.

\begin{proposition}\label{1}
 If $A$ is an SP-ring, $S\subseteq Reg(A)$ a multiplicative set and $P$ a prime ideal of $A$, then  $A_S$ \em (\em in particular $A_{(P)}$\em )\em\  is an SP-ring.
\end{proposition}
\begin{proof}
Let $H$ be a regular ideal of $A_S$. Then the regular ideal $J:=H\cap A$   is a product of radical ideals: $J=J_1\cdots J_n$. 
So each ideal $H_i=J_iA_S$ is radical and $H=H_1\cdots H_n$.
\end{proof}

\begin{remark}\label{13}
A  localization of an SP-ring at some prime ideal is not necessarily an SP-ring. Indeed, let $A=k[[x,y,z,v]]/(x^2,xy,xz,xv)$, where $k$ is a field. Then $A$ is a total quotient ring so an SP-ring. But $A_{(x,y,z)A}=k[[y,z,v]]_{(y,z)}$ is not an SP-ring, because it has dimension two and a Noetherian SP-domain is Dedekind, cf. \cite[Theorem 2.4]{VY}. Note also that $A/xA=k[[y,z,v]]$, so a factor ring of an SP-ring is not necessarily an SP-ring. 
\end{remark}

 Let $A$ be a ring and  $M$ an $A$-module. Recall that the {\em module idealization} $A(+)M$ of $M$ is the ring with abelian group $A\oplus M$ and multiplication given by $(a,x)(b,y)=(ab,ay+bx)$ for $a,b\in A$ and $x,y\in M$ (see \cite{AW}). Set $S=A-(Z(A)\cup Z(M))$  where $Z(A)$ (resp. $Z(M)$) are the set of zerodivisors on $A$ (resp. $M$).
 The next result translates the SP-ring property of $A(+)M$ in terms of $A$ and $M$. 

\begin{proposition}\label{6}
 With notations above,  $A(+)M$ is an SP-ring iff every ideal $I$ of $A$ not disjoint of $S$  is a product of radical ideals and $sM=M$ for each $s\in S$.
\end{proposition}
\begin{proof} 
Set $B=A(+)M$.
Following \cite{AW}, call an ideal $J$ of $B$ {\em homogeneous}
if it has the form $I \oplus N$ where $I$ is an ideal of $A$,
$N$ is a submodule of $M$ and $IM \subseteq N$. It can be checked directly (see also the last paragraph of page 12 in \cite{AW}) that the product of two homogeneous ideals is homogeneous.

$(\Rightarrow)$  By \cite[Theorem 3.2]{AW},  radical ideals of $B$ are homogeneous, more precisely they 
have the form $J\oplus M$ where $J$ is a radical ideal of $A$. Since $B$ is an SP-ring and, as noticed above, a product of homogeneous ideals is homogeneous, it follows that 
every regular ideal of $B$ is homogeneous, so $sM=M$ for each $s\in S$, cf. 
\cite[Theorem 3.9]{AW}. Let $J$ be an ideal of $A$ with $J\cap S\neq \emptyset$. By  \cite[Theorem 3.5]{AW}, $J\oplus M$ is a regular ideal of $B$, hence $J\oplus M=(J_1\oplus M)\cdots (J_n\oplus M)$ with each $J_i$ a radical ideal of $A$. We get $J=J_1\cdots J_n$.

$(\Leftarrow)$ Let $I$ be a regular ideal of $B$. By \cite[Theorem 3.9]{AW}, $I=J\oplus M$ for some ideal $J$ of $A$ with $J\cap S\neq \emptyset$. By our assumption, $J$ is a product of radical ideals of $A$, say $J=J_1\cdots J_n$. Since $sM=M$ for each $s\in S$, we get $J\oplus M=(J_1\oplus M)\cdots (J_n\oplus M)$ with each $J_i\oplus M$ a radical ideal of $B$, cf. \cite[Theorem 3.2]{AW}.
\end{proof}

\begin{example}
 Let $A$ be a ring and $\Gamma$ the set of regular ideals of $A$ which cannot be written as a product of radical ideals. For each $I\in \Gamma$ choose a maximal  ideal $N_I$ containing $I$, then set $M=\oplus_{I\in \Gamma}A/N_I$.
 Then  $A(+)M$ is an SP-ring. Indeed, set $S=A-(Z(A)\cup Z(M))$ and note that $Z(M)=\cup_{I\in \Gamma}N_I$. If $s\in S$, then $s(A/N_I)=A/N_I$ for every $I\in \Gamma$, so $sM=M$. Apply Proposition \ref{6}. 
\end{example}

Recall that a ring $A$ is an {\em N-ring} if $A_{(M)}$ is a discrete rank one Manis valuation ring  for each maximal regular ideal $M$ of $A$ (see \cite{Gf} and \cite{L}). A domain is an N-ring iff it is almost Dedekind. Our next result extends \cite[Theorem 2.4]{VY}.

\begin{theorem}\label{10}
If    $A$ is an  SP-ring, then $A$ is an N-ring.
\end{theorem}
\begin{proof}
We may suppose that $A$ is not a total quotient ring.
Let $M$ be a regular maximal ideal of $A$. By Proposition \ref{1}, $A_{(M)}$ is an SP-ring. Changing $A$ by $A_{(M)}$,  we may assume that $M$  is the only  regular maximal ideal of $A$. 
It suffices to show that $M$ is the only  regular prime ideal of $A$, for then  every regular ideal is a power of $M$ (as $A$ is an SP-ring), so
 $A$ is an N-ring by \cite[Theorem 1]{L}.
 Assume, to the contrary, that $A$ has a regular prime ideal  $P\subset M$. As $A$ is a Marot ring, we can pick a regular element $y\in M-P$. Shrinking $M$, we may assume that $M$ is minimal over $(P,y)$. Assume for the moment that the following two assertions hold.

$(1)$ $M\neq M^2$.

$(2)$ Every regular prime ideal $Q\subset M$ is contained in $M^2$.
\\
 By $(1)$ we can pick a regular element $\pi\in M-M^2$. From $(2)$ we get that every regular radical ideal except $M$ is contained in $M^2$. Then $M=\pi A$ because $A$ is an SP-ring. Let $s\in P$ be a regular element  and write $sA$ as a product of radical ideals $sA=H_1\cdots H_m$. So the prime ideal $P$ contains some (invertible) ideal $H_i:=H$.
 From $H\subseteq P\subset M=\pi A$, we get succesively: $H=\pi J$ for some invertible ideal $J\subseteq M$, $J^2\subseteq \pi J\subseteq H$, $J\subseteq \sqrt{H}=H=\pi J$,  $J=\pi J$.   Since $J$ is invertible,  we get the contradiction $\pi A=A$. It remains to prove $(1)$ and $(2)$.

{\em Proof of $(1)$.} 
 As $A$ is an SP-ring, it follows that  $(P,y)$ and $(P,y^2)$  are products of radical ideals, hence $(P,y)$ and $(P,y^2)$  are powers of $M$, 
 because $M$ is minimal over $(P,y)$. Then $M\neq M^2$, otherwise we get $(P,y)=M=(P,y^2)$ which leads to a contradiction after moding out by $P$, because $y\notin P$.

{\em Proof of $(2)$.} Assume, to the contrary, there exists a regular prime ideal
$Q\subset M$ such that $Q\not\subseteq M^2$. Pick a regular element  $z\in M-Q$. 
As $A$ is an SP-ring, it follows that $(Q,z^2)$ is a product of radical ideals, so $(Q,z^2)$ is a  radical ideal, because $Q\not\subseteq M^2$. We get  $(Q,z^2)= (Q,z)$ which gives a contradiction after moding out by $Q$. 
  
\end{proof}

Recall that a {\em regular-Noetherian ring}  is a ring whose regular ideals are finitely generated.  By \cite[Theorem 17]{Gf}, a regular-Noetherian N-ring is Dedekind. Thus we get:
\begin{corollary}\label{19}
 The regular-Noetherian SP-rings are exactly the Dedekind rings.
\end{corollary}


 We prepare the way for giving an extension of \cite[Theorem 2.1]{O} to SP-rings (see Theorem \ref{7}). Basically we adapt the ideas from the proof of \cite[Theorem 2.1]{O} to Marot rings.

\begin{lemma}\label{2}
Let $A$ be an N-ring and $I$ a regular ideal of $A$. Then $I$ is radical iff $I\not\subseteq M^2$ for each maximal ideal $M$. In particular, if $I$ is a   radical ideal, then every ideal containing $I$ is also radical.
\end{lemma}
\begin{proof}
Let $\Delta$ be the set of maximal ideals containing $I$. By \cite[Theorem 6.1]{H}, $I$ is radical iff $IA_{(M)}$ is radical for each $M\in\Delta$. Fix 
$M\in\Delta$ and set $B=A_{(M)}$. Since $M$ is regular and $A$ is an N-ring, $B$ is a  discrete rank one Manis valuation ring. Hence $IB$ is a radical ideal iff $IB=MB$ iff $IB\not\subseteq M^2B$ iff  $I\not\subseteq M^2$, because  $M^2B\cap A=M^2$. 
\end{proof}


\begin{lemma}\label{3}
Let   $A$ be an N-ring and $I$ a regular ideal such that $(IA_{(M)})^{-1}=I^{-1}A_{(M)}$ for each  maximal ideal $M$ containing $I$. Then 
$I$ is invertible (hence finitely generated).
\end{lemma}
\begin{proof}
Assume, to the contrary, that $II^{-1}$ is contained in some maximal ideal $M$ of $A$; so $I\subseteq M$. As $A$ is an N-ring, $IA_{(M)}=xA_{(M)}$ for some regular element $x\in I$. 
Using our assumption, we get $MA_{(M)}\supseteq II^{-1}A_{(M)}=IA_{(M)}(IA_{(M)})^{-1}=xA_{(M)}x^{-1}A_{(M)}=A_{(M)}$, a contradiction.
\end{proof}

Let $A$ be a ring, $E$ a nonempty subset of $A$ and $n\geq 1$. Denote by $V_n(E)$ the set $\{ P\in Spec(A)\mid E\subseteq P^n\}$. Note that when $A$ is an N-ring and $E$ contains a regular element, we have $V_n(E)=\{ M\in Max(A)\mid E\subseteq M^n\}$. 
Clearly, $V_{n+1}(E)\subseteq V_n(E)$. In particular, $V_1(E)$ is the usual Zariski closed set $V(E)$.

\begin{lemma}\label{4}
Let   $A$ be a ring and $I$ a proper  ideal of $A$ such that $V(I)\subseteq Max(A)$ and for every $M\in V(I)$ all powers of $M$ are distinct. The following are equivalent. 

$(a)$ $I$ is a product of radical ideals.

$(b)$  $I$ is a product  $J_1J_2\cdots J_n$ of radical ideals  with  $J_1\subseteq J_2\subseteq \cdots \subseteq J_n$.

$(c)$ The following three conditions hold

\quad\quad $(1)$ $IA_M$ is a power of $MA_M$ for each $M\in V(I)$, 

\quad\quad $(2)$ $V_n(I)$ is closed in $Spec(A)$ for each $n\geq 1$,

\quad\quad $(3)$ $V_n(I)$ is empty for  $n$ big enough. 
\end{lemma}
\begin{proof}
$(a)\Rightarrow (c).$ 
We have $I=J_1\cdots J_s$ with each $J_i$ a proper radical ideal. 
Let $n$ be an integer between $1$ and $s$ and pick  $M\in V(I)$. 
Since $J_i$ is radical, $I\subseteq J_i$ and $V(I)\subseteq Max(A)$, it follows that $J_iA_M=MA_M$ (resp. $J_iA_M=A_M$) if $J_i\subseteq M$ (resp. $J_i\not\subseteq M$). So $(1)$ holds. Moreover  $IA_{M}=J_1\cdots J_sA_M=M^tA_M$ where $t$ is the number of indices $i$ between $1$ and $s$ such that $J_i\subseteq M$.
Note that $M\in V_n(I)$ iff  $I\subseteq M^n$ iff $IA_M\subseteq  M^nA_{M}$, because $M^nA_{M}\cap A=M^n$.
So $M\in V_n(I)$ iff $J_i\subseteq M$ for at least $n$ indices $i$ iff $M$ contains the ideal  $H=\cap_{1\leq k_1<...<k_n\leq s} (J_{k_1}+\cdots +J_{k_n})$.
Thus $V_n(I)=V(H)$.
It is  clear that  $V_{s+1}(I)$ is empty.

$(c)\Rightarrow (b).$  
Let $n$ be the greateast integer $\geq 1$ such that $V_n(I)$ is nonempty. For each integer $k$ between $1$ and $n$, let $J_k$ be the unique radical ideal of $A$ such that $V(J_k)=V_k(I)$. Since $V_1(I)\supseteq \cdots \supseteq V_n(I)$, we get $I\subseteq J_1\subseteq \cdots \subseteq J_n$. Set $H=J_1\cdots J_n$. Let $M\in V(I)$. Since every $J_i$ is a radical ideal  containing $I$, we see that $HA_M=J_1\cdots J_nA_M=M^kA_M$ where $k$ is the greatest  index $i$ between $1$ and $n$ such that $J_i\subseteq M$. It follows that 
$M\in V(J_k)=V_k(I)$ and $M\notin V(J_{k+1})=V_{k+1}(I)$, in other words, $I\subseteq M^k$ and $I\not\subseteq M^{k+1}$. Consequently, $IA_M=M^kA_M$.  When $M$ is a maximal ideal not containing $I$, we get $IA_M=HA_M=A_M$. Thus we verified locally that $I=H$. The implication $(b)\Rightarrow (a)$ is clear.
\end{proof}

\begin{corollary}\label{5}
 If   $A$ is a ring and ${(I_\alpha)}_\alpha$  a family of proper ideals of $A$ such that each  ${I_\alpha}$ satisfies  the equivalent conditions of Lemma \ref{4}, then so does  $\sum_\alpha I_\alpha$.
\end{corollary}
\begin{proof}
The ideal  $H=\sum_\alpha I_\alpha$ satisfies part $(c)$ of Lemma \ref{4}, because $V_n(H)=\cap_\alpha V_n(I_\alpha)$.
\end{proof}

 The following theorem extends \cite[Theorem 2.1]{O} to SP-rings. Recall that every 
   SP-ring is an N-ring, cf. Theorem \ref{10}.

\begin{theorem}\label{7} 
For an N-ring $A$, the following assertions are equivalent.

$(i)$ $A$ is an SP-ring.

$(ii)$ Every  regular maximal ideal contains some regular finitely generated radical ideal.

$(iii)$ The radical of a regular finitely generated ideal is  finitely generated. 

$(iv)$ For every $x\in Reg(A)$, $xA$ is a product of radical ideals.

$(v)$ For every $x\in Reg(A)$ and $n\geq 1$, the set $V_n(x)$ is closed in $Spec(A)$ and $V_m(x)$ is empty for some $m\geq 1$.

$(vi)$ For every regular proper ideal $I$ and $n\geq 1$, the set $V_n(I)$ is closed in $Spec(A)$ and $V_m(I)$ is empty for some $m\geq 1$.

$(vii)$ Every regular proper ideal $I$ can be factorized as $I=J_1J_2\cdots J_n$ with radical ideals $J_1\subseteq J_2\subseteq \cdots \subseteq J_n$.

$(viii)$ Every regular proper ideal $I$ can be factorized uniquely as $I=J_1J_2\cdots J_n$ with radical ideals $J_1\subseteq J_2\subseteq \cdots \subseteq J_n$.
\end{theorem}
\begin{proof}
$(i)\Rightarrow (ii)$. Let $M$ be a regular maximal ideal, pick $x\in M\cap Reg(A)$ and write as a product $J_1J_2\cdots J_n$ of radical ideals, cf. $(i)$.
Then each $J_i$ is invertible (hence finitely generated) and $M$ contains one of them.

$(ii)\Rightarrow (iii)$. Let $I$ be a regular finitely generated ideal and $M$ a maximal ideal containing $I$. By $(ii)$, there exists some regular finitely generated radical ideal $C$ contained in $M$. By Lemma  \ref{2}, $J_1=I+C$ is a radical ideal. Since $I\subseteq J_1$ are finitely generated ideals (hence invertible), we obtain $I=J_1B_1$ for some invertible ideal $B_1$. We repeat this argument until we get 
$I=J_1\cdots J_nB_n$ with each $J_i$ an invertible radical ideal contained in $M$ and $B_n$ is an ideal not contained in $M$. This can be done because $J_1\cdots J_n\subseteq M^nA_{(M)}$ and the intersection $\cap_k M^kA_{(M)}$ consists of zero-divisors, cf. \cite[Theorem 2]{L}. Since each $J_i$ is invertible and radical,  $H=J_1\cap \cdots \cap J_n$ is a finitely generated radical ideal. We get
$\sqrt{I}=H\cap \sqrt{B_n}\subseteq H$ and $\sqrt{I}A_{(M)}=HA_{(M)}$, because $B_n\not\subseteq M$. 
We have $(\sqrt{I}A_{(M)})^{-1}=(HA_{(M)})^{-1}=H^{-1}A_{(M)}\subseteq (\sqrt{I})^{-1}A_{(M)}\subseteq (\sqrt{I}A_{(M)})^{-1}$, so
$(\sqrt{I}A_{(M)})^{-1}=(\sqrt{I})^{-1}A_{(M)}$. By Lemma \ref{3}, $\sqrt{I}$ is finitely generated.

$(iii)\Rightarrow (iv)$.  
Deny. Let  $x\in Reg(A)$ such that $I=xA$ is not a product of radical ideals.
By $(iii)$, $J_1=\sqrt{I}$ is finitely generated hence invertible. Since $I\subseteq J_1$, it follows that $I=J_1B_1$ for some invertible ideal $B_1\neq A$. We repeat this argument to write $I=J_1J_2B_2$ with $J_2=\sqrt{B_1}$ and some  invertible ideal $B_2\neq A$. Since $J_1$ is finitely generated, we get $J_1^k\subseteq I\subseteq J_2$ for some $k\geq 1$, so $J_1\subseteq J_2$ because $J_2$ is radical.  Since  $I$ is not a product of radical ideals, this process continues indefinitely to produce an infinite increasing  sequence of proper radical ideals $J_1\subseteq J_2\subseteq\ ...$ such that $I\subseteq J_1\cdots J_n$ for each $n\geq 1$. If $M$ is a maximal ideal containing $\cup_n J_n$, then 
$I\subseteq \cap_n M^nA_{(M)}$ which is a contradiction because the intersection $\cap_n M^nA_{(M)}$ consists of zero-divisors, cf. \cite[Theorem 2]{L}.

The equivalencies $(iv)\Leftrightarrow (v)$ and $(vi)\Leftrightarrow (vii)$ follow from Lemma \ref{4}. Also $(v)\Leftrightarrow (vi)$ follows from Lemma \ref{4} and Corollary \ref{5} (note that every regular ideal satisfies $(c_1)$ of Lemma \ref{4} because $A$ is an N-ring).

$(vii)\Rightarrow (viii)$. Let $I$ be a regular proper ideal having two factorizations $I=J_1J_2\cdots J_n=H_1H_2\cdots H_m$
with radical ideals $J_1\subseteq J_2\subseteq \cdots \subseteq J_n$ and $H_1\subseteq H_2\subseteq \cdots \subseteq H_m$. We get $\sqrt{I}=J_1\cap \cdots \cap J_n=J_1$ and similarly $\sqrt{I}=H_1$, hence $J_1=H_1$. By \cite[Theorem 15]{Gf}, regular ideal $J_1=H_1$ is cancellable, so $J_2\cdots J_n=H_2\cdots H_m$ and we complete by induction.
The implication $(viii)\Rightarrow (i)$ is clear.
\end{proof}


Recall that a {\em Pr\" ufer ring} is a ring whose regular finitely generated ideals are invertible (see \cite{Gf}).

\begin{remark}
Let $A$ be a ring.
By \cite[Theorem 3]{L},  $A$ is an N-ring iff $(\alpha)$ $A$ is Pr\" ufer ring with regular primes maximal and $(\beta)$ no regular maximal ideal is idempotent.
As in the domain case (see \cite[Corollary 2.2]{O}), the following two  assertions can be added to the list of equivalent conditions in Theorem \ref{7}.

$(ii')$ $A$ satisfies $(\alpha)$ and condition $(ii)$ of Theorem \ref{7}.

$(iv')$ $A$ satisfies $(\alpha)$ and condition $(iv)$ of Theorem \ref{7}.
\\
Indeed, assuming  $(ii')$, it suffices to prove that $(\beta)$ holds.
Deny, so let $M$ be an idempotent regular maximal ideal  of $A$. By $(ii)$ of Theorem \ref{7}, $M$ contains some regular radical finitely generated ideal $I$. We get $IA_M=MA_M=M^2A_M=(IA_M)^2$, so Nakayama's Lemma gives $IA_M=0$, a contradiction. Finally,  to see that $(iv')$ implies $(ii')$, let $M$ be a   maximal ideal of $A$ containing a regular element $r$. By $(iv')$, $rA$ is a product of radical invertible (hence finitely generated) ideals, so $M$ (being prime) contains one of the factors.
\end{remark}

We end this section by giving two corollaries of Theorem \ref{7}.
We recall the following well-known fact: 

\begin{lemma}\label{20}
 Let $K$ be a field of characteristic zero, $f\in K[X]-\{0\}$, $\pi\in K[X]$ an irreducible polynomial and $k\geq 1$ an integer. Then $\pi^k$ divides $f$ iff $\pi$ divides  $f,f',...,f^{(k-1)}$, where  $f^{(j)}$ is the $j$th derivative of $f$.
\end{lemma}

\begin{corollary}\label{18}
Let $A$ be a von Neuman regular ring containing $\mathbb{Q}$. Then $A[X]$ is an SP-ring.
\end{corollary}
\begin{proof}
 By  \cite[Theorem 7.5]{H},  $B=A[X]$ is a Marot ring. By \cite[Theorem 6]{A}, $B$ is an almost multiplication ring (see definition recalled before  Theorem \ref{11}), so $B$ is an N-ring, cf. \cite[Theorem 6]{L}. 
We show that $B$ satisfies condition $(v)$ in Theorem \ref{7}.
Let $f\in Reg(B)$. 
Since $A$ is von Neuman regular, the content of $f$ equals $eA$ for some idempotent $e\in A$, hence $e=1$ because $f$ is regular. 
Let $N\in Max(B)$. Then $N=(M,g)$ for some $M\in Max(A)$ and $g\in B$ which is irreducible modulo $M$. As $M=M^2$, we get $N^k=(M,g^k)$ for each $k\geq 1$.
Since the content of $f$ is $A$, the image of $f$ in $(A/M)[X]$ is nonzero, so $f\not\in  N^{n+1}$, in other words, $V_{n+1}(f)=\emptyset$, where $n$ is the degree of $f$.
Now let $k$ be an integer between $1$ and $n$. Since $\mathbb{Q}\subseteq A$, it follows that $K=A/M$ is a field of characteristic zero. 
If $h\in A[X]$, denote its image in $K[X]$ by $\bar{h}$.
 Then $f\in N^k$ iff $\bar{g}^k$ divides $\bar{f}$ iff $\bar{g}$ divides 
$\bar{f},\overline{f'},...,\overline{f^{(k-1)}}$, where  $f^{(j)}$ is the $j$th derivative of $f$, cf. Lemma \ref{20}. Thus  $V_k(f)$ is the closed set $V(f,f',...,f^{(k-1)})$.

\end{proof}

\begin{corollary}\label{12}
Let $B$ be the direct product  of some family of rings  $(A_i)_{i\in I}$.  Then $B$ is an SP-ring iff  all the $A_i$'s are SP-rings and at most  finitely many of them are not total quotient rings.
\end{corollary}
\begin{proof}
$(\Rightarrow)$ It is easy to see that all the $A_i$'s are SP-rings. Assume, to the contrary, that  infinitely many of them are not total quotient rings. Considering a direct factor of $B$, we may assume that $I=\mathbb{N}$ and no $A_i$ is a total quotient ring. 
In each $A_i$ select a regular  nonunit  element $a_i$ 
and set $x=(a_1,a_2^2,a_3^3,...)\in B$. Then $\sqrt{xB}=\cup_{n\geq 1}y_nB$ where $y_n=(a_1^{\lceil{1/n}\rceil}, a_2^{\lceil{2/n}\rceil}, a_3^{\lceil{3/n}\rceil}, ...)$. Since  $y_nB\subset y_{n+1}B$ for each $n\geq 1$, it follows that $\sqrt{xB}$ is not finitely generated, a contradiction cf. part $(iii)$ of Theorem \ref{7}.
$(\Leftarrow)$ We combine the following two facts: $(1)$ a direct product of total quotient rings is a total quotient ring hence an SP-ring; $(2)$ a direct product of finitely many SP-rings is an SP-ring.
\end{proof}


\section{SSP-rings}

 Call a ring $A$ a {\em strongly SP-ring (SSP-ring)} if every ideal is a product of radical ideals. Clearly, SSP-domains are exactly SP-domains and every SSP-ring is an SP-ring. A von Neuman regular ring has all ideals radical, so it is a trivial example of SSP-ring. A  ZPI ring (i.e. a ring whose proper ideals are products of prime ideals, see \cite[page 205]{LM}) is another example of an SSP-ring. The ring $B=\mathbb{Z}_2[x,y]/(x^2,xy,y^2)$ is zero-dimensional (hence an SP-ring) but not SSP because $Spec(B)=\{(x,y)B\}$  and $xB$ is not a power of $(x,y)B$.

\begin{proposition}\label{8}
Let $A$ and $B$ be  rings.

$(a)$ If $A$ is an SSP-ring, then every factor ring (resp. fraction ring) of $A$ is SSP.

$(b)$ If $A$ is an SP-ring and $H$ a regular  ideal of $A$, then  $A/H$ is an SSP-ring.

$(c)$ $A\times B$ is an SSP-ring iff $A$ and $B$ are SSP-rings.
\end{proposition}
\begin{proof}
The  assertions $(a)$ and $(c)$ are easy to check. $(b)$. Assume that $A$ be an SP-ring and $H$ a regular proper ideal of $A$. Let $I\supseteq H$ be an ideal of $A$. Then $I$ is a regular ideal, so  $I=J_1\cdots J_n$ with each $J_i$ a radical ideal, because  $A$ is an SP-ring.
We get $I/H=(J_1/H)\cdots (J_n/H)$ with each $J_i/H$ a radical ideal.
\end{proof}

Recall that a ring $R$ is a {\em special primary ring} if $Spec(R)=\{M\}$  and each proper ideal of $R$ is a power of $M$.

\begin{lemma}\label{9}
Let $A$ be a local SSP-ring with maximal ideal $M$.
If $A$ is a one-dimensional domain (resp. a zero-dimensional ring), then $A$ is a discrete rank one   valuation domain (resp. a special primary ring).
\end{lemma}
\begin{proof}
It suffices to remark that $M$ is the only nonzero radical ideal (resp. radical ideal) of $A$.
\end{proof}

Recall that an {\em almost multiplication ring} is a ring whose localizations at its prime ideals are discrete rank one  valuation domains or special primary rings. 
The following result is an analogue of Theorem \ref{10} (the proof being rather similar).

\begin{theorem}\label{11}
If    $A$ is an  SSP-ring, then $A$ is an almost multiplication ring.
\end{theorem}
\begin{proof}
Let $M$ be a   maximal ideal of $A$. By Proposition \ref{8}, $A_M$ is an SSP-ring, hence   we may assume that $A$ is local with maximal ideal  $M$. 
By Lemma \ref{9}, it suffices to show that $A$ has no other nonzero prime ideal except $M$.
Deny, so assume that $A$ has a nonzero prime ideal  $P\subset M$. Pick an element $y\in M-P$. Shrinking $M$, we may assume that $M$ is minimal over $(P,y)$. Assume for the moment that the following two assertions hold.

$(1)$ $M\neq M^2$.

$(2)$ Every  prime ideal $Q\subset M$ is contained in $M^2$.
\\
By $(1)$ we can pick an element $\pi\in M-M^2$. From $(2)$ we get that every radical ideal except $M$ is contained in $M^2$. Then $M=\pi A$ because $A$ is an SSP-ring.
Let $s\in P$ be a nonzero element  and write $sA$ as a product of radical ideals $H_1\cdots H_m$.  As $P$ is prime, it contains some  $H_i$. From $P\subset M=\pi A$, we get succesively: $H_i=\pi J$ for some ideal $J\subseteq M$, $J^2\subseteq H_i$, $J\subseteq \sqrt{H_i}=H_i$, $J=\pi J=H_i=\pi H_i$.  Combining the equalities $sA=H_1\cdots H_m$ and $H_i=\pi H_i$, we get $sA=\pi sA$, so $s(1-\pi t)=0$ for some $t\in A$. Since $1-\pi t$ is a unit, we have $s=0$ which is  a contradiction.
It remains to prove $(1)$ and $(2)$.

{\em Proof of $(1)$.} As $A$ is an SSP-ring, $y\in M-P$ and $M$ is minimal over $(P,y)$,  it follows that $(P,y)$ and $(P,y^2)$  are powers of $M$. Then $M\neq M^2$, otherwise we get $(P,y)=M=(P,y^2)$ which leads to a contradiction after moding out by $P$.

{\em Proof of $(2)$.} Deny. So there exists a  prime ideal
$Q\subset M$ such that $Q\not\subseteq M^2$. Pick an element  $z\in M-Q$. 
As $A$ is an SSP-ring, it follows that $(Q,z^2)$ is a product of radical ideals, so $(Q,z^2)$ is a  radical ideal, because $Q\not\subseteq M^2$. We get  $(Q,z^2)= (Q,z)$ which gives a contradiction after moding out by $Q$. 

\end{proof}


The SSP-ring with finitely many minimal prime ideals are easy to describe.

\begin{corollary}\label{15}
For a ring $A$, the following statements are equivalent:

$(a)$ $A$ is an SSP-ring which has only finitely many minimal prime ideals,

$(b)$ $A$ is an SSP-ring whose minimal prime ideals  are finitely generated,

$(c)$  $A$ is a finite direct product of special primary rings and SP-domains.
\end{corollary}
\begin{proof}
$(a) \Leftrightarrow (b)$ follows from Theorem \ref{11} and \cite[Theorem 5]{A}. 
$(c)\Rightarrow (a)$ follows from part $(c)$ of Proposition \ref{8}.
$(a)\Rightarrow (c)$. By Theorem \ref{11} and \cite[Theorem 5]{A}, $A$ is a finite direct product of special primary rings and (almost Dedekind) domains. Apply part $(c)$ of Proposition \ref{8}.
\end{proof}

We extend the well-known fact that Noetherian SP-domains are exactly Dedekind domains.

\begin{corollary}\label{14}
A ring $A$ is a Noetherian SSP-ring iff $A$ is a ZPI-ring.
\end{corollary}
\begin{proof}
$(\Rightarrow)$ By Theorem \ref{11}, $A$ is a Noetherian almost multiplication ring, hence a ZPI-ring, cf. \cite[Theorem 13]{M}. 
$(\Leftarrow)$ A ZPI-ring is Noetherian (cf. \cite[Theorem 9.10]{LM}) and  an SSP-ring.
\end{proof}

Recall that an $A$-module $E$ is a {\em multiplication module} if for each submodule $F$ of  $E$, $F = IE$ for some ideal $I$ of $A$.

\begin{proposition}\label{16}
 If $A$ is a von Neuman regular ring and $E$ a multiplication $A$-module, then $A(+)E$ is an SSP ring. 
\end{proposition}
\begin{proof}
 Denote $A(+)E$ by $B$. As $A$ is von Neuman regular, every ideal of $A$ is radical, so radical ideals of $B$ have the form $I\oplus E$ with $I$ an ideal of $A$, cf. \cite[Theorem 3.2]{AW}. By \cite[Theorem 3.3]{AW}, every ideal of $B$ is homogeneous. Indeed, if  $a\in A$, we have $a=a(ab)$ for some $b\in A$ (because $A$ is von Neuman regular), so $aE=abE$.  Let $H$ be a (homogeneous) ideal of $B$.  As $E$ is a multiplication module, $H=I\oplus JE$ for some ideals $I$ and $J$ of $A$ such that $IE\subseteq JE$. Changing $J$ by $I+J$, we may assume that $I\subseteq J$. As $A$ is von Neuman regular, we have $I=IJ$. 
 Now $I\oplus E$ and $J\oplus E$ are radical ideals and $(I\oplus E)(J\oplus E)=IJ\oplus (I+J)E=H$.
\end{proof}

By Theorem \ref{11}, an SSP ring is  an almost multiplication ring. Our last result characterizes the SSP property among a particular class of almost multiplication rings. For a module $E$, we denote by $Supp(E)$ its support.

\begin{theorem}\label{17}
Let $A$ be a  ring  whose localizations at the maximal ideals are discrete rank one   valuation domains (hence $A$ is an almost multiplication ring).
Then $A$ is an SSP-ring iff for every  ideal $I$ of $A$ the following two conditions hold 

$(1)$ for each $n\geq 1$, $V_n(I)\cap Supp(I)$ is closed in the induced topology of $Supp(I)$,

$(2)$ $V_n(I)\cap Supp(I)$ is empty for  $n$ big enough. 
\end{theorem}
\begin{proof}
$(\Rightarrow)$ Let $I$ be an  ideal  of $A$. We may assume that $0\neq I\neq A$.
Write $I$ as a product of proper radical ideals $J_1\cdots J_s$. 
Let $n$ be an integer between $1$ and $s$ and pick  $M\in V(I)\cap Supp(I)$. We have  $I\subseteq M^n$ iff $IA_M\subseteq  M^nA_{M}$, because $M^nA_{M}\cap A=M^n$.
Note that  $J_iA_M$ is a  radical ideal of $A_M$ and $MA_M$ is the only  proper nonzero radical ideal of $A_{M}$. So $IA_{M}=J_1\cdots J_sA_M=M^tA_M$ where $t$ is the number of indices $i$ between $1$ and $s$ such that $J_i\subseteq M$. So $M\in V_n(I)\cap Supp(I)$ iff $J_i\subseteq M$ for at least $n$ indices $i$. 
Therefore $V_n(I)\cap Supp(I)=V(H)\cap Supp(I)$ where $H$ is the ideal 
$\cap_{1\leq k_1<...<k_n\leq s} (J_{k_1}+\cdots +J_{k_n})$.
It is also clear that  $V_{s+1}(I)\cap Supp(I)$ is empty.

$(\Leftarrow)$  Let $I$ be a proper ideal  of $A$. As $A$ is reduced, we may assume $I\neq 0$.
Let $n$ be the greateast integer $\geq 1$ such that $V_n(I)\cap Supp(I)$ is nonempty. For each integer $i$ between $2$ and $n$, let $H_i$ be an  ideal of $A$ such that $V(H_i)\cap Supp(I)=V_i(I)\cap Supp(I)$. 
Note that $V(I+H_i)\cap Supp(I)=V(I)\cap V(H_i)\cap Supp(I)=V(I)\cap V_i(I)\cap Supp(I)=V_i(I)\cap Supp(I)$. 
Changing $H_i$ by $\sqrt{I+H_i}$, we may assume that each $H_i$ is a radical ideal containing $I$. We set $H_1=\sqrt{I}$.
We check locally that  $I=H_1\cdots H_n$. Let $M\in Max(A)$. We examine the  cases: $(1)$ $M\not\supseteq I$, $(2)$ $M\in V(I)-Supp(I)$, $(3)$ $M\in V(I) \cap Supp(I)$.
 In case $(1)$ we have  $IA_M=A_M=(H_1\cdots H_n)A_M$ because every $H_i$ contains $I$.  In case $(2)$ we have $IA_M=0=(H_1\cdots H_n)A_M$, because  $A_M$ is a domain so  $H_1A_M=\sqrt{IA_M}=0$.  
 Assume that we are in case $(3)$. We have $IA_M\neq 0$, so $IA_M=M^kA_M$ for some $k\geq 1$ because $A_M$ is a discrete rank one   valuation domain. Hence $I\subseteq M^k$ and $I\not\subseteq M^{k+1}$, that is, $M\in V_k(I)-V_{k+1}(I)$. We have $H_i\subseteq M$ iff 
 $H_iA_M=MA_M$ (since $H_i$ is radical)
 iff $i\leq k$. Consequently, $IA_M=M^kA_M=(H_1\cdots H_n)A_M$.
\end{proof}

\
\
\

{\bf Acknowledgements.} 
We thank the referee whose suggestions improved our paper.
The first author is highly grateful to ASSMS GC University Lahore, Pakistan in supporting and facilitating this research.
The second author gratefully acknowledges the warm
hospitality of Abdus Salam School of Mathematical Sciences GC University Lahore during his  visits in the period 2006-2016.

\end{document}